\documentclass{amsart}

\title[Dynamics of two-resonant]{Dynamics of two-resonant biholomorphisms}
\author{Jasmin Raissy, Liz Vivas}
\address{J. Raissy: Institut de Math\'ematiques de Toulouse, Universit\'e Paul Sabatier, 118 route de Narbonne, 31062 Toulouse, France.} \email{jraissy@math.univ-tlse.fr}
\address{L. Vivas: Institute Henri Poincar\'e, 11 rue Pierre et Marie Curie, 75231 Paris, France. \newline
Department of Mathematics, Purdue University, 150 N. University Street, West Lafayette, IN 47907-2067.} \email{lvivas@math.purdue.edu}

\date{November \the\day, \the\year}

\usepackage{amscd}
\usepackage{color}

\usepackage[margin=0.9in]{geometry}

\newtheorem{theorem}{Theorem}

\newtheorem{proposition}{Proposition}
\newtheorem{corollary}{Corollary}

\theoremstyle{definition}
\newtheorem{defin}{Definition}

\theoremstyle{remark}
\newtheorem{remark}{Remark}

\newcommand{\NN}{\mathbb{N}}

\newcommand{\CC}{\mathbb{C}}
\newcommand{\PP}{\mathbb{P}}

\def\a{{\alpha}}
\def\b{{\beta}}

\def\l{{\lambda}}

\def\1#1{\overline{#1}}
\def\2#1{\widetilde{#1}}
\def\3#1{\widehat{#1}}
\def\4#1{\mathbb{#1}}
\def\5#1{\frak{#1}}
\def\6#1{{\mathcal{#1}}}
\def\7#1{{\bf{#1}}}

\def\C{{\4C}}

\def\N{{\4N}}

\def\Q{{\4Q}}

\let\no=\noindent

\let\me=\medskip
\let\sm=\smallskip

\def\Dif{{\sf Diff}(\C^n,O)}
\def\Diff{{\sf Diff}}

\def\id{{\sf id}}

\def\Re{{\sf Re}\,}

\def\Arg{{\rm Arg}\,}

\def\res{{\rm Res}\,}

\def\phi{\varphi}

\begin{document}

\begin{abstract}
In this paper we study the existence of basins of attraction for germs of 2-resonant biholomorphisms of $\C^n$ fixing a point, that is germs such that the eigenvalues of the differential at the fixed point have a 2 dimensional family of resonances.
\end{abstract}

\maketitle

\section{introduction}

Given a germ $F$ of biholomorphism of $\C^n$ at a fixed point, that we may choose (without loss of generality) to be the origin $O$, and with diagonalizable differential at it, we are interested in the dynamics of $F$ near the origin, that is in the dynamical behavior of the sequence of the iterates of $F$ in a (sufficiently small) neighborhood of $O$. 

Since the local dynamics of a germ is invariant under conjugacy, one of the main tools to study it is the classification under holomorphic, topological or formal conjugacy of the germ. Moreover, it is natural to search for ``special representatives'' having an easier to understand dynamics. 

The most natural candidate is the linear germ, but, especially for $n\ge 2$, even the formal conjugacy to the linear term is not guaranteed due to the possible presence of {\it resonances} between the eigenvalues $\lambda_1,\dots,\lambda_n$, i.e., the relations $\l_h=\prod_{j=1}^n \l_j^{q_j}$ for $1\leq h\leq n$, with $q_1,\dots,q_n\in\N$, and $q_1+\cdots+q_n\ge 2$. However, for non-linearizable germs, another natural candidate (at least in the formal category) is given by {\it Poincar\'e-Dulac normal forms} (see \cite[Chapter 5]{Ar} and \cite{Ab2} for surveys). The study of the convergence of Poincar\'e-Dulac normal forms is made difficult, in most of the cases, by {\it small divisors problems} and by the not uniqueness of such normal forms. On the other hand, it is always possible to holomorphically conjugate to a germ in {\it approximated} Poincar\'e-Dulac normal form, that is in normal form up to a given order. Recently, these approximated normal forms have been successfully used to infer the dynamical behavior of a germ.

Furthermore, whereas almost everything is known for the local holomorphic dynamics in the one-dimensio\-nal case (see for example the surveys \cite{Ab2} and \cite {Br}), the study of local dynamics of non holomorphically linearizable germs is still far from being complete, and only some cases have been studied in detail. The simplest case is the {\it attracting} (resp. {\it repelling}) case, i.e., when all the eigenvalues of the linear term have modulus strictly less than 1 (resp. strictly larger than 1), where the dynamics of the germ is always depicted by its linear term at the origin. 

Another interesting case is the {\it tangent to the identity} one, that is germs whose linear part is the identity. Results on the existence of basins of attraction were proved by \'Ecalle \cite{Ec}, Hakim \cite{Ha, Ha1} (see also \cite{Ari}), and, more recently by the second author \cite{Vi}. In the tangent to the identity case, every multi-index is indeed resonant, so Poincar\'e-Dulac normal forms are not of much help. 

Recently, holomorphic Poincar\'e-Dulac normal forms have been linked to the existence of invariant holomorphic foliations with some additional properties, \cite{Ra}. Furthermore, Bracci and Zaitsev \cite{BZ} and Bracci, Zaitsev and the first author \cite{BRZ}, studied the dynamics of {\it multi-resonant} germs, i.e., germs whose resonances are generated over $\N$ by a finite number of $\Q$-linearly independent multi-indices, through the study of the dynamics induced by their approximated Poincar\'e-Dulac normal forms on the almost invariant holomorphic foliations (see \cite{BZ} and \cite{BRZ} for details, and the next section for a brief summary of the main results and definitions).

\sm In this paper, we study local dynamics of $2$-resonant, non-linearizable germs with diagonalizable linear term. In the rest of the paper, and without mentioning it explicitly, we shall consider only germs of biholomorphims whose differential at $O$ is diagonalizable. 

Following \cite{BRZ}, we say that a germ is {\it $2$-resonant with respect to the first $r$ eigenvalues $\{\l_1,\ldots, \l_r\}$} if there exist two linearly independent multi-indices $P,Q\in \N^r\times\{0\}^{n-r}$, such that all resonances $\l_s=\prod_{j=1}^n \l_j^{\beta_j}$ for $1\leq s\leq r$ are precisely of the form $(\beta_1,\ldots, \beta_n)=k P+ h Q+ e_s$ with $k,h\in \N$, $e_s=(0,\ldots, 0,1,0\ldots, 0)$ with $1$ in the $s$-th coordinate.

Given a germ as above, we can consider the map $\pi\colon\C^n\to\C^2$, $\pi(z)=(z^P, z^Q)$, and the {\it parabolic shadow} $f$ of $F$ as it is defined in \cite{BRZ} (see also section 2), which is a tangent to the identity germ of $(\C^2, O)$ describing roughly the dynamics induced by approximated Poincar\'e-Dulac normal forms of $F$ on the almost-invariant foliation $\{z\in\C^n : \pi(z)=\hbox{const}\}$. 

The results in \cite{BRZ} were obtained using the dynamical properties of the parabolic shadow $f=\id + H_{k_0 +1}$ (namely the existence of basins of attraction centered along a {\it fully-attracting} {\it non-degenerate} characteristic direction ensured by the works of Hakim \cite{Ha, Ha1}) together with an attracting condition on the germ $F$ with respect to the projection, called {\it parabolically-attracting} introduced in \cite{BRZ}. 

In \cite{Vi} the second author generalized Hakim's results for maps tangent to the identity in $\CC^2$. She proved the existence of basins of attraction centered along {\it irregular} or {\it Fuchsian-attracting} {\it degenerate} characteristic directions, as well as centered along {\it irregular} non-degenerate characteristic directions, which are in particular not fully-attracting (see also \cite{La}).

It could then seem possible to use all the results of \cite{Vi}, together with the parabolically-attracting condition of \cite{BRZ}, to 
obtain basins of attraction for $2$-resonant germs whose parabolic shadow has a degenerate or an irregular non-degenerate characteristic direction. 

Our first result is that this is not possible for the degenerate characteristic directions. More precisely, we prove that for a map whose parabolic shadow has a degenerate characteristic direction, then it can not be also parabolically-attracting. We provide examples of germs with parabolic shadows having a basin along a degenerate characteristic direction but with no basins of attraction (see section 4). On the other hand, in our second result, we prove the existence of basins when the parabolic shadow has an irregular non-degenerate characteristic direction and the map $F$ is parabolically-attracting.

Before stating our main results we need to recall a couple of definitions. 

Let $v\in\C^2\setminus\{0\}$ be a non-degenerate characteristic direction for $f$ in the sense of \'Ecalle \cite{Ec} and Hakim \cite{Ha}, i.e., $H_{k_0+1}(v) = c v$ for some $c\in\C\setminus\{0\}$. It is clear that if $v$ is a characteristic direction, then any scalar multiple of $v$ will also be a characteristic direction, hence the direction $[v]$ in $\PP^1$ is well determined. In the rest of the paper we will refer by $v$ or $[v]$ as the characteristic direction. By multiplying by an appropriate constant, we can normalize the direction so that $H_{k_0+1}(v) = -(1/k_0) v$ (we say $v$ is {\it normalized} as in \cite{BRZ}). Following the definitions introduced by Abate and Tovena \cite{AT2} (see also section 2), we say that $F$ is {\it $(f,v)$-irregular-nondegenerate}, if $v$ is an irregular non-degenerate characteristic direction for $f$. We say that $F$ is {\it $(f,v)$-parabolically-attracting} with respect to $\{\l_1,\ldots, \l_r\}$ if $v$ is a normalized characteristic direction of $f$, and
\begin{equation}
\Re\left(\sum_{k_1+k_2=k_0\atop (k_1,k_2)\in\N^2} {\frac{a_{(k_1,k_2), j}}{\l_j}} v_1^{k_1}v_2^{k_2}\right) <0 \quad j=1,\ldots, r.
\end{equation}
Such conditions are invariant in the sense that if they hold for a parabolic shadow, then they hold for any other parabolic shadow of $F$.
We say that $F$ is {\it irregular-nondegenerate} if it is $(f,v)$-irregular-nondegenerate with respect to some parabolic shadow $f$ and some normalized characteristic direction $v\in \C^2$. We also say that $F$ is {\it parabolically-attracting} if it is $(f,v)$-parabolically-attracting with respect to some parabolic shadow $f$ and some normalized non-degenerate characteristic direction $v\in \C^2$.

Our main results are then the following:

\begin{proposition}\label{mainprop}
Let $F\in\Dif$ be $m$-resonant with respect to the first $r\le n$ eigenvalues, and let $f$ be a parabolic shadow of $F$. If $v\in\C^m$ is any representative of a degenerate characteristic direction for $f$, then $F$ can not be $(f,v)$-parabolically-attracting.
\end{proposition}

\begin{theorem}\label{mainthmintro}
Let $F\in\Dif$ be $2$-resonant with respect to the eigenvalues
$\{\l_1,\ldots\l_r\}$. Assume that
$|\l_j|=1$ for $j=1, \dots, r$ and $|\l_j|<1$ for
$j=r+1,\ldots, n$. If $F$ is irregular-nondegenerate and parabolically-attracting, then
there exists a basin of attraction
having $O$ at the boundary.
\end{theorem}

\sm
The strategy of the proof of Theorem \ref{mainthmintro} uses the same ideas as in \cite{BRZ}. Indeed, we use the parabolic shadow $f$ to obtain a basin $B$ in $\CC^2$. Then we define $\tilde{B} \subset \pi^{-1}(B)$ in $\CC^n$ so that $F(\tilde{B}) \subset \tilde{B}$. In order to make sure that $\tilde{B}$ is $F$-invariant, we choose the first $r$ coordinates so that their projection by $\pi$ is in $B$, and we choose the other coordinates in such a way that their modulus is less than a power of the projection coordinates. Using the hypothesis of parabolically-attracting we are able to see that the modulus of the coordinates will go exponentially fast to $0$.  

This result complements the corresponding result in \cite{BRZ} for the case $m=2$, and moreover, from it we can also deduce the non-necessity of the hypothesis of {\it fully-attracting} (see section 2) for the characteristic direction of the parabolic shadow.

\me

The plan of the paper is the following. In section 2 we briefly recall the basic definitions, the result that we need from the dynamics of tangent to the identity germs, and we prove Proposition \ref{mainprop}. In section 3 we prove Theorem \ref{mainthmintro}. In section 4 we apply Theorem \ref{mainthmintro} to a few examplee, we provide an example of germ with no basins of attraction, and we make some observations on partial cases for holomorphically normalizable germs.

\section{preliminaries}

As stated in the introduction, we shall follow the strategy of the recent papers \cite{BZ} and \cite{BRZ}, where results about the dynamics of germs of biholomorphisms tangent to the identity in $\CC$, and in $\CC^m$, were used to prove the existence of attracting basins for one-resonant, respectively multi-resonant, germs satisfying additional conditions. We shall use the same approach together with a recent result on the dynamics of germs tangent to the identity in $\CC^2$ to prove results on the dynamics of 2-resonant germs satisfying additional conditions.

We start by briefly recalling the definition of multi-resonant germ, and we refer the reader to the very thorough study in \cite{BRZ}. 

Given $\{\lambda_1,\ldots,
\lambda_n\}$ a set of complex non-zero numbers, a {\em
resonance} is a pair $(j,L)$, where $j\in\{1,\ldots, n\}$ and
$L=(l_1,\ldots, l_n)\in\N^n$ is a multi-index with
$|L|:=\sum_{h=1}^n l_h \ge2$ such that $\lambda_j=\lambda^L$
(where $\lambda^L:=\lambda_1^{l_1}\cdots \lambda_n^{l_n}$). We
shall use the notation
$$
\res_j(\lambda) := \{Q\in\N^n : |Q|\ge 2, \lambda^Q = \lambda_j\}.
$$
With a slight abuse of notation, we denote by $e_j=(0,\ldots, 0,
1, 0,\ldots, 0)$ both the multi-index with $1$ at the $j$-th
position and $0$ elsewhere, and the vector with the same entries
in $\C^n$. 

Let $F$ be in $\Dif$, and let $\lambda_1,\ldots, \lambda_n$ be
the eigenvalues of the differential $dF_O$. We say that $F$ is
{\it $m$-resonant with respect to the first $r$ eigenvalues}
$\l_{1},\ldots,\l_{r}$ $(1\le r\le n)$ if there exist $m$
multi-indices $P^1,\dots,P^m\in \N^{r}\times\{O\}^{n-r}$ linearly independent
over $\Q$, so that the  resonances $(j,L)$ with $1\le j\le r$
are precisely of the form
\begin{equation}
L = e_j + k_1 P^1 + \cdots + k_m P^m
\end{equation}
with $k_1,\dots, k_m\in\N$ and $k_{1}  + \cdots + k_m\ge 1$. The vectors $P^1,\dots, P^m$ are called {\it generators over $\N$} of the resonances of $F$.
We call $F$ {\it multi-resonant with respect to the first $r$ eigenvalues}  if it is $m$-resonant
with respect to these eigenvalues for some $m\ge1$.

We focus our attention to the case $m=2$. 

Given $F$ a 2-resonant germ with respect to the first $r$ eigenvalues $\l_1,\ldots,\l_r$, and any given $l$ large, by the theorem of Poincar\'e-Dulac (see \cite{Ar}) we can holomorphically conjugate $F$ to a map of the form
\begin{equation}
\widetilde F(z) = Dz + \sum_{s=1}^r\sum_{|k_1 P+ k_2 Q|\ge 2 \atop (k_1, k_2)\in\N^2} a_{(k_1, k_2),s} z^{ k_1 P+ k_2 Q} z_s e_s +  \sum_{s=r+1}^n R_s(z)e_s + O(\|z\|^l),
\end{equation}
where $D={\rm Diag}(\lambda_1,\dots, \lambda_n)$, $R_s(z)=O(\|z\|^2)$ for $s=r+1,\ldots,n$, and $P,Q$ are the ordered generators over $\N$ of the resonances.
The {\it weighted order} of $F$ is the minimal $k_0=k_1+k_2\in\N\setminus\{O\}$ such that the coefficient $a_{(k_1,k_2),s}$ of $\2{F}$ is non-zero for some $1\le s\le r$. It was proved in \cite{BRZ} that such a definition is well-posed.
We recall that the weighted order of $F$ is $+\infty$ if and only if $F$ is formally linearizable in the first $r$ coordinates.

\no Define $$G(z) = Dz + \sum_{s=1}^r\sum_{|k_1 P+ k_2 Q|\ge 2 \atop (k_1, k_2)\in\N^2} a_{(k_1, k_2),s} z^{ k_1 P+ k_2 Q} z_s e_s. $$ 
Using the projection map $\pi(z) = (z^P,z^Q)$, we obtain the following commuting diagram:
\[ 
\begin{CD} 
\CC^n @>G>> \CC^n\\ 
@V\pi VV  @V\pi VV\\
\CC^2 @>h>> \CC^2 
\end{CD} \] 
where $h$ is a tangent to the identity germ in $\CC^2$ of the form $h(u) = u + H_{k_0 + 1}(u)+O(\|u\|^{k_0+2})$, where
\begin{equation}\label{matricehakim}
H_{k_0 + 1}(u) =\left(\begin{array}{c}
				\displaystyle u_1\sum_{k_1+k_2 = k_0} \left({p_1}\frac{a_{(k_1,k_2), 1}}{\l_1} + \cdots + {p_r}\frac{a_{(k_1,k_2), r}}{\l_r}\right)u_1^{k_1}u_2^{k_2}\\
				\displaystyle u_2\sum_{k_1+k_2 = k_0} \left({q_1}\frac{a_{(k_1,k_2), 1}}{\l_1} + \cdots + {q_r}\frac{a_{(k_1,k_2), r}}{\l_r}\right)u_1^{k_1}u_2^{k_2}
				\end{array}\right),
\end{equation}
and, by construction, it approximates up to a given finite order the action of $F$ on the almost invariant foliation $\{z\in\C^n : \pi(z)=\hbox{const}\}$. We shall call $f(u)=u+H_{k_0+1}(u)$ a {\it parabolic shadow} of $F$.
\sm

We shall also need a few definitions for tangent to the identity germs, that we recall here.
Given
\begin{equation}\label{h_0}
h(u) := u + H_{k_0+1}(u) + O(\|u\|^{k_0+2}),
\end{equation}
a germ at $O$ of biholomorphism of $(\C^m, O)$, $m\ge2$, tangent to the identity, where  $H_{k_0+1}$ is the first non-zero term in the homogeneous expansion of $h$, and $k_0\ge 1$, we call the number $k_0+1\ge 2$ the {\it order} of $h$.

Several important results about basins for tangent to the identity germs have been obtained in the last decade. We refer the reader to \cite{BRZ} where several results have been explained and applied to the existence of parabolic basins for resonant germs, and we recall here only the result that will be relevant to prove our main theorem.

\begin{defin}
Let $h\in\Diff(\C^m, O)$ be of the form \eqref{h_0}.  A {\it
characteristic direction} for $h$ is a non-zero vector
$v\in\C^m\setminus\{O\}$ such that $H_{k_0+1}(v)=\lambda v$ for
some $\lambda\in\C$. If $\lambda=O$, $v$ is a {\it
degenerate} characteristic direction; otherwise, (that is, if
$\lambda\ne 0$) $v$ is {\it non-degenerate}.
\end{defin}

In the case $m=2$, examples of basins along non-degenerate not fully-attracting and degenerate characteristic directions have been shown recently (see \cite{Ab2}, \cite{V}), inspired also on the classification introduced in the recent paper \cite{AT2}, where Abate and Tovena studied the dynamics of the time 1-map of homogeneous holomorphic vector fields $\C^2$. 
We shall use the result obtained in \cite{Vi} for tangent to the identity germs in $\C^2$ with non-degenerate characteristic direction, but before stating it we need to explain the classification given by Abate and Tovena \cite{AT2}.

Writing $u=(x,y)\in\C^2$, our map has the form:
\begin{equation}
h(u) := u + H_{k_0+1}(u) + O(\|u\|^{k_0+2}),~~\hbox{where}~~H_{k_0+1}(x,y) =  (Q_1(x,y),Q_2(x,y))
\end{equation}
where both $Q_1$ and $Q_2$ have homogeneous degree $k_0+1$. Using the above definitions, we clearly have that $(1,u_0)$ is a characteristic direction when $Q_2(1,u_0) = u_0Q_1(1,u_0)$. Equivalently if we define $g_2(u) = Q_2(1,u) - uQ_1(1,u)$, then $(1,u_0)$ is a characteristic direction if $g_2(u_0) = 0$. Define $g_1(u) = Q_1(1,u)$. 

Let $\mu_j(u_0)\in \NN$ be the order of vanishing of $g_j$ at $u_0$. Then we shall say that the direction given by $(1,u_0)$ is:

\begin{itemize}
\item an {\it apparent} characteristic direction if $\mu_2(u_0) < \mu_1(u_0) + 1$;
\item a {\it Fuchsian} characteristic direction if $\mu_2(u_0) = \mu_1(u_0) + 1$; and
\item an {\it irregular} characteristic direction if $\mu_2(u_0) > \mu_1(u_0) + 1$.
\end{itemize}
The index of a direction $i_{(1,u_0)}$ is defined as the residue $\textrm{Res}_{u=u_0}\frac{g_1(u)}{g_2(u)}$. 

\sm Using such a classification, in \cite{Vi} the second author proved that if a tangent to the identity germ in $\Diff(\C^2,O)$ has an irregular characteristic direction, or a Fuchsian characteristic direction such that $\Re i_{(1,u_0)} > 1/\mu_1(u_0)$, then there exists a parabolic basin of attraction centered in the direction. Note that Fuchsian non-degenerate characteristic directions such that $\Re i_{(1,u_0)} > 1/\mu_1(u_0)$ are exactly {\it fully-attracting} non-degenerate characteristic directions as defined in \cite{BRZ}, i.e., the real parts of the directors (see \cite{Ha} or \cite{Ari} for the precise definition) are strictly positive, and hence those characteristic directions satisfy the hypotheses of Hakim's result \cite{Ha}. In particular, for non-degenerate characteristic directions, we shall use the following result that complements the result of Hakim in dimension $2$.

\begin{theorem}[\cite{Vi}]\label{irregular}
Let $h\in\Diff(\C^2, O)$ be a tangent to the identity germ of biholomorphism. If $[v]=[1:u_0]$ is an irregular non-degenerate characteristic direction for $h$, then there exists a parabolic domain $D$ for $h$ centered in the direction $[v]$. Moreover, there exists a Fatou coordinate $\psi: D \to \CC$ such that $\psi(h(z)) = \psi(z) + 1$.
\end{theorem}

In \cite{BRZ}, the existence of parabolic basins of attraction depended not only on the existence of a basin for the parabolic shadow of the germ, but also on another ``attracting'' condition controlling the dynamics on the fibers of $\pi\colon\C^n\to \C^m$, which is the following:

\begin{defin}\label{def-parab_0}
Let $F\in\Dif$ be $2$-resonant with respect to $\l_{1}, \ldots,\l_{r}$, with $P, Q$ being the ordered generators over $\N$ of the resonances.
Let $k_0<+\infty$ be the weighted order of $F$. Let $f=\id + H_{k_0+1}$, with $H_{k_0+1}$ as in \eqref{matricehakim}, be a parabolic shadow of $F$.
We say that $F$ is {\it $(f,v)$-parabolically-attracting} with respect to $\{\l_1,\ldots, \l_r\}$ if $v$ is a normalized non-degenerate characteristic direction of $f$, (i.e., $H_{k_0+1}(v) = -(1/k_0)v$), and
\begin{equation}\label{parabolic_0}
\Re\left(\sum_{k_1+k_2=k_0\atop (k_1,k_2)\in\N^2} {\frac{a_{(k_1,k_2), j}}{\l_j}} v_1^{k_1}v_2^{k_2}\right) <0 \quad j=1,\ldots, r.
\end{equation}
We say that $F$ is {\it $(f,v)$-partially parabolically-attracting of order $s$} if there exists $1\le s<r$ such that \eqref{parabolic_0} is satisfied only for $j=1,\dots, s$.
We say that $F$ is {\it parabolically-attracting} if $F$ is $(f,v)$-parabolically-attracting with respect to some parabolic shadow $f$ and some $v\in \C^2$ a normalized non-degenerate characteristic direction.
We say that $F$ is {\it partially parabolically-attracting of order $s$} if  $F$ is $(f,v)$-partially parabolically-attracting of order $s$ for some $1\le s<r$ with respect to some parabolic shadow $f$ and some $v\in \C^2$ normalized non-degenerate characteristic direction.
\end{defin}

As mentioned in the introduction, 
one could think that it is possible to use all the results of \cite{Vi}, together with the attracting condition just defined, to obtain basins of attraction also for $2$-resonant germs whose parabolic shadow has a degenerate characteristic direction. However, this is not true, and we can provide in section~4 examples of germs with parabolic shadow having a basin along a degenerate characteristic direction but with no basins of attraction. Moreover, we have the following result, that holds in general for $m$-resonant germs with $m\ge2$ with respect to the condition of parabolic-attracting introduced in \cite[Definition 3.12]{BRZ}.

\begin{proposition}\label{Propo}
Let $F\in\Dif$ be $m$-resonant with respect to the first $r\le n$ eigenvalues, and let $f$ be the parabolic shadow of $F$. If $v\in\C^m$ is any representative of a degenerate characteristic direction for $f$, then $F$ cannot be $(f,v)$-parabolically-attracting.
\end{proposition}

\begin{proof}
Let $k_0\ge 1$ be the weighted order of $F$ and let $P^1, \dots, P^m\in\N^r\times\{0\}^{n-r}$ be the ordered generators of the resonances of $F$. Since $F$ is $m$-resonant with respect to the first $r$ eigenvalues $\l_1,\ldots,\l_r$, then by the Theorem of Poincar\'e-Dulac we can holomorphically conjugate it to a map of the form
\begin{equation}
\widetilde F(z) = Dz + \sum_{s=1}^r\sum_{|K|\ge k_0 \atop K\in\N^m} a_{K,s} z^{k_1 P^1+\cdots+ k_m P^m} z_s e_s +  \sum_{s=r+1}^n R_s(z)e_s + O(\|z\|^l),
\end{equation}
where $D={\rm Diag}(\lambda_1,\dots, \lambda_n)$, and $R_s(z)=O(\|z\|^2)$ for $s=r+1,\ldots,n$. 
Hence, for the parabolic shadow $f(u) = u + H_{k_0+1}(u)$ we have
$$
H_{k_0+1, j} (u) = u_j \sum_{|K|=k_0} \left(p_1^j \frac{a_{K,1}}{\lambda_1} + \cdots+ p_r^j \frac{a_{K,r}}{\lambda_r}\right) u^K
$$
for all $j=1,\dots, r$.

Let $v\in\C^m\setminus\{O\}$ be a degenerate characteristic direction for $f$, i.e., such that $H_{k_0+1}(v) =0$, and assume by contradiction that
$$
\Re\left(\sum_{|K|=k_0} {\frac{a_{K, s}}{\l_s}} v^K\right) <0 \quad \forall s=1,\ldots, r.
$$ 
Let $j\in \{1,\dots, m\}$ be such that $v_j\ne0$. Then, since $v$ is degenerate, we have
\begin{equation}
\begin{aligned}
0&=
\sum_{|K|=k_0} \left(p_1^j \frac{a_{K,1}}{\lambda_1} + \cdots+ p_r^j \frac{a_{K,r}}{\lambda_r}\right) v^K\\
&= p_1^j\left(\sum_{|K|=k_0}\frac{a_{K,1}}{\lambda_1}v^K\right) + \cdots +p_r^j \left(\sum_{|K|=k_0}\frac{a_{K,r}}{\lambda_r} v^K\right).
\end{aligned}
\end{equation}
Therefore
$$
p_1^j\Re\left(\sum_{|K|=k_0}\frac{a_{K,1}}{\lambda_1}v^K\right) + \cdots +p_r^j \Re\left(\sum_{|K|=k_0}\frac{a_{K,r}}{\lambda_r} v^K\right) = 0 
$$
which is impossible since $P^j$ is a generator, and hence $P^j\in \N^n$ and $|P^j|\ge 1$, and we are assuming $$\displaystyle{\Re\left(\sum_{|K|=k_0}\frac{a_{K,s}}{\lambda_s}v^K\right)<0}$$ for all ${s=1,\dots, r}$. 
\end{proof}

We now give the definition that we use to state our result. 

\begin{defin}\label{basic-def_0}
Let $F\in\Dif$ be $2$-resonant with respect to $\l_{1}, \ldots,\l_{r}$, with $P, Q$ being the ordered generators over $\N$ of the resonances.
Let $k_0<+\infty$ be the weighted order of $F$. Let $f=\id + H_{k_0+1}$, with $H_{k_0+1}$ as in \eqref{matricehakim}, be a parabolic shadow of $F$.
We say that $F$ is {\it $(f,v)$-irregular-nondegenerate} if $v\in \C^2$ is an irregular normalized non-degenerate characteristic direction for $f$. We say that $F$ is {\it irregular-nondegenerate} if $F$ is $(f,v)$-irregular-nondegenerate with respect to some parabolic shadow $f$ and some $v\in \C^2$ normalized non-degenerate characteristic direction.
\end{defin}

\begin{remark}
It is immediate to check that Definitions \ref{def-parab_0} and \ref{basic-def_0} are well posed: if $F$ is $(f,v)$-irregular-nondegenerate, or $(f,v)$-parabolically-attracting, with respect to a parabolic shadow $f$ and a normalized characteristic direction $v\in \C^2$, then it is so with respect to any parabolic shadow (for the corresponding normalized characteristic direction $\tilde{v}$). 
\end{remark}

\section{non-degenerate 2-resonant germs}

\begin{theorem}\label{mainthm1}
Let $F\in\Dif$ be $2$-resonant with respect to the  eigenvalues
$\{\l_1,\ldots\l_r\}$ and of weighted order $k_0$. Assume that
$|\l_j|=1$ for $j=1, \dots, r$ and $|\l_j|<1$ for
$j=r+1,\ldots, n$. If $F$ is irregular-non-degenerate and parabolically-attracting, then
there exists a basin of attraction
having $O$ at the boundary.
\end{theorem}

\begin{proof}
We follow the same strategy as in \cite{BZ} and \cite{BRZ}.
Let $P$ and $Q$ be the ordered generators over $\N$ of the resonances of $F$.
Up to biholomorphic conjugation, we can assume that $F(z) = (F_1(z),\dots, F_n(z))$ is of the form
$$
\begin{aligned}
&F_j(z) = \lambda_j z_j \Bigg(1 + \sum_{k_0\le k_1+k_2\le k_l \atop (k_1,k_2)\in\N^2}  \frac{a_{(k_1,k_2),j}}{\lambda_j} z^{ k_1 P+ k_2 Q} \Bigg) + O\left(\|z\|^{l+1}\right), \quad j=1,\ldots, r,\\
&F_{j}(z)=\l_{j} z_{j} + O(\|z\|^{2}) , \quad\quad\quad\quad\quad\quad \quad\quad \quad\quad\quad\quad\quad\quad \quad\quad\quad\quad\;\, j=r+1,\ldots, n,
\end{aligned}
$$
where
$$
k_l:= \max \{k_1+k_2 :  |k_1 P+ k_2 Q| \le l \}.
$$
 
We consider the map $\pi\colon (\C^n,0)\to (\C^2,0)$ defined by $\pi(z) = u := (z^P, z^Q)$. Then, for $j=1,\dots, r$, we can write
$$
F_j(z) =   G_{j}(u,z) + O\left(\|z\|^{l+1}\right), \quad  G_{j}(u,z):= \lambda_j z_j
\Bigg(1 + \sum_{k_0\le k_1+k_2\le k_l \atop (k_1,k_2)\in\N^2}  \frac{a_{(k_1,k_2),j}}{\lambda_j} u_1^{k_1} u_2^{k_2} \Bigg).
$$
The composition $\phi:= \pi\circ F$ can be written as
$$
\phi(z)=\Phi(u,z): =  \1\Phi(u) + g(z), \quad \1\Phi(u)=  u + H_{k_0 + 1}(u) + h(u),
$$
where $\1\Phi$ is induced by $G$ via $\pi\circ G=\1\Phi\circ \pi$, $H_{k_0+1}(u)$ has the form \eqref{matricehakim}, that is
\begin{equation*}
H_{k_0 + 1}(u) =\left(\begin{array}{c}
				\displaystyle u_1\sum_{k_1+k_2 = k_0} \left({p_1}\frac{a_{(k_1,k_2), 1}}{\l_1} + \cdots + {p_r}\frac{a_{(k_1,k_2), r}}{\l_r}\right)u_1^{k_1}u_2^{k_2}\\
				\displaystyle u_2\sum_{k_1+k_2 = k_0} \left({q_1}\frac{a_{(k_1,k_2), 1}}{\l_1} + \cdots + {q_r}\frac{a_{(k_1,k_2), r}}{\l_r}\right)u_1^{k_1}u_2^{k_2}
				\end{array}\right),
\end{equation*}
we have $h(u) = O(\|u\|^{k_0+2})$, and $g(z) = O(\|z\|^{l+1})$.

Thanks to our hypotheses, the parabolic shadow $u\mapsto u + H_{k_0 + 1}(u)$, has an irregular non-degenerate characteristic direction $[v]$. Therefore, Theorem \ref{irregular} implies that $\1\Phi$ has an attracting basin $B$ of parabolic type at the origin {\it centered} along the normalized characteristic direction $v\in\C^2$, with a Fatou coordinate. We will construct a basin of attraction $\2V\subset \C^n$ for $F$ in such a way that $\2V$ is projected into $B$ via $\pi$. 

For the sake of simplicity, we shall use linear coordinates $(s,t)\in \C^2$ where the normalized characteristic direction is $v = (1,0)$. Then the condition on $F$ to be parabolically-attracting translates on:
\begin{equation}
\Re\Bigg({\frac{a_{(k_0,0), j}}{\l_j}} \Bigg) <0 \quad j=1,\ldots, r.
\end{equation}

After blowing-up $\C^2$ at the origin as in \cite{Vi}, in the local chart
centered at $v$, we may assume the lifting of the map $\Phi$ to be of the form:
\begin{equation}
\begin{aligned}
& s_1 = \Phi_1(s,t,z) = s - \frac{1}{k_0} s^{k_0 + 1} + h_1(s,t) + g_1(z),\\
& t_1 = \Phi_2(s,t,z) = t - \frac{1}{n-1} s^{k_0}t^n + h_2(s,t) + s^{-1}g_2(z)
\end{aligned}
\end{equation}
with $h_1= O(s^{k_0+2}, s^{k_0+1}t)$, $h_2= O(s^{k_0+1}, s^{k_0}t^{n+1})$, and $g_1(z),g_2(z) = O(\|z\|^{l+1})$.

Thanks to Theorem \ref{irregular}, $\1\Phi(s,t)$ has an attracting basin $B$ along $v$.
We shall prove that, fixed $\beta>0$ small enough, the set
$$
\2V:=\{z\in\C^n : |z_j|< |s|^{k_0\beta} ~\hbox{for}~j=1,\dots, n, \; \pi (z):= (s, t)\in B\}
$$
is a basin of attraction for $F$. First of all, taking $\beta>0$ sufficiently small, it is easy to see that $\widetilde V$ is  an open non-empty set  of $\C^n$ and $O\in\partial \widetilde V$.

Next, we prove that $\widetilde V$ is $F$-invariant. Let $z\in \widetilde V$ and let $u=\pi(z)$.
Considering the change of coordinates $(x,y)= \psi(s,t) :=(s^{-k_0}, t^{-(n-1)})$ on a suitable open set, with the origin on its boundary, we have that $(s,t)\in B$ if and only if $(x,y)\in V$, where
$$
V=V_{R, N, \theta}:= \{(x,y)\in\C^2: \Re(x)>R, |\Arg(x)|<k_0 \theta, \Re(y)>R, |y|^N<|x|\},
$$
for $R,N$ large enough, and $\theta$ small. Fix $0<\delta<1/2$ and $0<c'<c$. Thanks to the parabolically-attracting hypothesis, there exists $\theta>0$ such that
\begin{equation}\label{oneba}
\left|1+ \sum_{k_1+k_2= k_0 \atop (k_1,k_2)\in\N^2} \frac{a_{(k_1,k_2),j}}{\l_j}  u_1^{k_1}u_2^{k_2} \right| \le 1 - 2{c }|s|^{k_0},
\end{equation}
for all $u\in B$. 
We choose $\beta>0$ such that
\begin{equation}\label{betabisb_0}
\beta(\delta +1 ) -c'<0
\end{equation}
and we can choose $l>1$ so that
\begin{equation}
\beta (l+1) \geq 4.
\end{equation}

Since $z \in \2 V$, then we have the following estimates for $g_1$ and $s^{-1}g_2$:
\begin{eqnarray*} 
\|s^{-1}g_2(z)\| < C |s|^{-1}\|z\|^{l+1} \leq C' |s|^{-1}|s|^{k_0\beta(l+1)} = C|s|^{k_0\beta(l+1)-1}.
\end{eqnarray*}
Similarly for $\|g_1(z)\|$; and we conclude that $\|g_1(z)\|,\|s^{-1}g_2(z)\|= O(|s|^{k_0+2})$.

In the coordinates $(x,y)= \psi(s,t) :=(s^{-k_0}, t^{-(n-1)})$ we have
\begin{equation}
\begin{aligned}
&x_1= x + 1 + \nu_1(x,y,z),\\
&y_1 = y + \frac{1}{x} + \nu_2(x,y,z),
\end{aligned}
\end{equation}
with
\begin{equation}
\begin{aligned}
&\nu_1(x,y,z) = O\left(\frac{1}{x^{1/k_0}}, \frac{1}{y^{1/(n-1)}}\right),\\
&\nu_2(x,y,z) = O\left(\frac{y^{\frac{n}{n-1}}}{x^{(k_0+1)/k_0}}, \frac{1}{xy^{1/(n-1)}}\right)
\end{aligned}
\end{equation}
If $R$ and $N$ are sufficiently large, for any $z\in \widetilde V$ we have $|\nu_1(x, y,z)|<\delta<1/2$. Therefore, we have $\Re(x_1)>R$ and $|\Arg(x_1)|<|\Arg(x)|<\theta$. Arguing as in the computations of \cite[p. 9]{Vi}, we obtain $\Re(y_1)>R$, and $|y_1|^N<|x_1|$, and so we proved that if $u=\pi(z)\in B$, then $u_1:=\pi(F(z))\in B$.

Now, given $z\in \widetilde V$, we have to estimate $|F_j(z)|$ for $j=1,\dots, n$.
To estimate the components  $F_j$ for $j=r+1,\ldots, n$, we can argue exactly as in \cite{BRZ} and we refer the reader to \cite[p. 17]{BRZ} for the proof. 
Hence, we have
\begin{equation}\label{primastima}
|F_{j}(z)|<  |s_1|^{k_0\beta}, \quad j=r+1,\ldots,n.
\end{equation}

For the other coordinates, we have
$$
F_j(z) = \lambda_j z_j \Bigg(1+ \sum_{k_1+k_2= k_0 \atop (k_1,k_2)\in\N^2} \frac{a_{(k_1,k_2),j}}{\l_j}  u_1^{k_1} u_2^{k_2} + f_j(u) \Bigg)+ g_j(z),
$$
with $f_{j}=O(\|u\|^{k_{0}+1})$ and $g_{j}=O(\|z\|^{l+1})$, for $j=1,\ldots,r$.
Thanks to \eqref{oneba} we have
$$
\left| 1+ \sum_{k_1+k_2= k_0 \atop (k_1,k_2)\in\N^2} \frac{a_{(k_1,k_2),j}}{\l_j}  u_1^{k_1} u_2^{k_2}  + f_j(u)\right|  \le 1 - \frac{c}{|x|}.
$$
Moreover, if $z\in \widetilde V$ and $R$ is sufficiently large, we have, for a suitable $D>0$,
$$
|g_j(z)|\le D\|z\|^{l+1} < \frac{D}{|x|^{\b(l+1)}}.
$$
Therefore for $j=1,\ldots, r$
\begin{equation}\label{Fjb_0}
\begin{aligned}
|F_j(z)|
&\le |\lambda_j| |x|^{-\b} \left( 1 - \frac{c }{|x|} \right) +\frac{D}{|x|^{\b(l+1)}}, \\
&\le \left( 1 - \frac{c }{|x|}+ \frac{D}{|x|^{\b l}} \right) |x|^{-\b}.
\end{aligned}
\end{equation}
Hence, if $R$ is
sufficiently large, we have
$$
p(x) : = 1 - \frac{c }{|x|}+ \frac{D}{|x|^{\b l}} <1.
$$
Now we claim that
\begin{equation}\label{stimatwob0}
|x_1| \le p(x)^{-1/\beta}|x|.
\end{equation}
Indeed, since we have
$$
x_1 = x + 1 + \nu_1(x, y, z),
$$
with $|\nu_1(x, y, z)|\le \delta$, we obtain
$$
\frac{|x_1|}{|x|} = \frac{|x + 1 + \nu_1(x, y, z)|}{|x|}
\le  1 + \frac{1}{|x|} + \frac{|\nu_1(x,y,z)|}{|x|}
\le 1 + \frac{1+ \delta }{|x|}.
$$
On the other hand, by our choice of $0<c'<c$ and taking $R$
sufficiently large, we have
$$
\left(1-\frac{c'}{|x|}\right)^{-1/\beta} \le p(x)^{-1/\beta},
$$
and hence, in order to prove \eqref{stimatwob0}, we just need to
check that
$$
1 + \frac{1+ \delta }{|x|} \le \left(1-\frac{c'}{|x|}\right)^{-1/\beta}.
$$
But
$$
\left(1-\frac{c'}{|x|}\right)^{-1/\beta} = 1 + \frac{1}{\beta}\frac{c'}{|x|} + O\left(\frac{1}{|x|^{2}}\right),
$$
and since \eqref{betabisb_0} ensures that $\delta + 1 -c'/\beta<0$, if $R$ is sufficiently large, \eqref{stimatwob0} holds, and the claim is proved.
Therefore, thanks to \eqref{Fjb_0} we have
\begin{equation}
|F_j(z)| < |s_1|^{k_0\beta}, \quad j=1,\dots, r,
\end{equation}
which together with \eqref{primastima} implies
 $F(\widetilde V)\subseteq \widetilde V$.

Finally, setting inductively $u^{(l)}=(s^{(l)},t^{(l)}): = \pi(F^{\circ (l-1)}(z))$, and denoting by $\rho_j\colon\C^n\to\C$ the projection $\rho_j(z)= z_j$, we obtain
$$
\left|\rho_j\circ F^{\circ l}(z)\right|\le \left|s^{(l)}\right|^{k_0\beta}
$$
for all $z\in \widetilde V$, implying that $F^{\circ l}(z)\to O$ as
$l \to +\infty$. This proves that $\widetilde V$ is a basin of
attraction of $F$ at $O$.
\end{proof}

\section{examples and remarks on partial cases}

\subsection{Example of parabolically attracting germs}
In \cite[Section 5.1]{BRZ} is shown a family of $2$-resonant germs in $\C^3$ which are $(f,v_1)$-attracting-nondegenerate and $(f,v_2)$-attracting-nondegenerate (with $f$ a parabolic shadow, and $v_1, v_2$ two different normalized non-degenerate characteristic directions for $f$), and which are $(f,v_1)$-parabolically-attracting but not $(f,v_2)$-parabolically-attracting. Here we shall weaken certain conditions and see that the existence of basins is still satisfied by using our Theorem \ref{mainthm1}.

Let $P^1=(2,3,0)$ and $P^2=(0,2,5)$. Let $\l_1,\l_2,\l_3\in \C^\ast$ be of modulus $1$ such that relations are generated by $\l_1^2\l_2^3=1$ and $\l_2^2\l_3^5=1$. It is easy to see that  $\l_j=\l^L$ for $L\in \N^3$, $|L|\geq 2$, $j=1,2,3$ if and only if $L=k_1P^1+k_2P^2+e_j$ for some $k_1,k_2\in \N$.

Let $F$ be of the form
\begin{equation*}
F_j(z)=\l_j z_j (1+b_{e_1,j}z^{P^1}+b_{e_2,j}z^{P^2})\quad j=1,2,3.
\end{equation*}
Then $F$ is $2$-resonant with respect to $\{\l_1,\l_2,\l_3\}$ and of weighted order $k_0=1$. A parabolic shadow of $F$ is $f(u)= u+H_2(u)$ where
\[
H_2(u)=\left(\begin{array}{c}
				\displaystyle u_1\left[ (2b_{e_1,1} + 3b_{e_1,2})u_1+(2b_{e_2,1} + 3b_{e_2,2})u_2\right]\\
				\displaystyle u_1\left[ (2b_{e_1,2} + 5b_{e_1,3})u_1+(2b_{e_2,2} + 5b_{e_2,3})u_2\right]
				\end{array}\right).
\]
The directions $[1:0]$ and $[0:1]$ are characteristic directions, and as seen in \cite{BRZ}, imposing $2b_{e_1,1} + 3b_{e_1,2}=-1, \quad 2b_{e_2,2} + 5b_{e_2,3}=-1,$ it follows that the two directions are non-degenerate characteristic directions for $f$. Furthermore setting $2b_{e_2,1} + 3b_{e_2,2}=-p,\quad 2b_{e_1,2} + 5b_{e_1,3}=-q$ with $q,p>1$ we have that $(1,0)$ and $(0,1)$ are normalized fully-attracting non-degenerate characteristic directions for $f$, and hence $F$ is $(f,(1,0))$-attracting-non-degenerate and $(f,(0,1))$-attracting-non-degenerate.

However, if we impose the following condition:
\begin{equation}\label{dima2}
2b_{e_1,1} + 3b_{e_1,2}=-1, \quad 2b_{e_2,1} + 3b_{e_2,2}=-p =-1
\end{equation}
we obtain that $(1,0)$ is an irregular normalized characteristic direction for $f$, and hence $F$ is $(f,(1,0))$-irregular. 

Analogously, imposing
\begin{equation}\label{dima3}
2b_{e_2,2} + 5b_{e_2,3}=-1, \quad 2b_{e_1,2} + 5b_{e_1,3}=-q =-1
\end{equation}
we have that $(0,1)$ is an irregular normalized characteristic direction for $f$, hence $F$ is $(f,(0,1))$-irregular. 

The condition for $F$ to be $(f,(1,0))$-parabolically-attracting it is as before:
\begin{equation}\label{dima5}
\Re b_{e_1,j}<0 \quad j=1,2,3\,,
\end{equation}
whereas $F$ is $(f,(0,1))$-parabolically-attracting if and only if
\begin{equation}\label{dima6}
\Re b_{e_2,j}<0 \quad j=1,2,3\,.
\end{equation}

\subsection{Example of $2$-resonant degenerate germ with no basins of attraction.} 
Let $P^1=(2,3,0)$, $P^2=(0,2,5)$, and let $\l_1,\l_2,\l_3\in \C^\ast$ be of modulus $1$ such that relations are generated by $\l_1^2\l_2^3=1$ and $\l_2^2\l_3^5=1$, as in the previous example.

Let $F$ be of the form
\begin{equation*}
F(z)=\left(\lambda_1 z_1(1 + z^{P^2})\,,\, \lambda_2 z_2\,,\, \lambda_3 z_3\left(1+ {3\over 5} z^{P^2}\right)\right).
\end{equation*}
Then $F$ is $2$-resonant with respect to $\{\l_1,\l_2,\l_3\}$ and of weighted order $k_0=1$. It is clear that $F$ has no basins of attraction. Indeed, for any point $w\ne0$ in a neighborhood of zero such that $w_2\ne 0$ we have $(F^{\circ\ell}(w))_2= \lambda_2^\ell w_2$ and then their orbits cannot converge to $O$, and for any point $w\ne0$ in a neighborhood of zero such that $w_2= 0$ we have $(F^{\circ\ell}(w))= (\lambda_1^\ell w_1, 0, \lambda_3^\ell w_3)$ and hence also their orbits cannot converge to $O$.
However, parabolic shadows of $F$ have a basin of attraction. In fact, a parabolic shadow of $F$ is $f(u)= u+H_2(u)$ where
\[
H_2(u)=\left(2u_1u_2,3u_2^2\right).
\]
The direction $[1:0]$ is a degenerate Fuchsian characteristic direction satisfying the hypotheses of \cite[Theorem 2]{Vi}, whereas $[0:1]$ is a non-degenerate characteristic direction whose director has strictly negative real part. Therefore, $f$ has a basin of attraction centered in $[1:0]$ but it cannot have a basin centered in $[0:1]$ (see \cite{Ha1} or \cite{Ari}). Note that $F$ is not $(f, (1,0))$-parabolically-attracting.

\subsection{Remarks on partial cases.} We end this section showing that for holomorphically normalizable germs, by weakening the hypothesis of parabolically attractiveness, we can still get information on the dynamics, namely the existence of parabolic manifolds. We recall that a {\it parabolic manifold} $P$  of dimension $1\leq s\leq
n$ for $F\in\Diff(\C^n, 0)$ is the biholomorphic image of a simply connected open set in $\C^s$ such that $O\in \partial P$, $F(P)\subset P$ and $\lim_{\ell\to\infty}F^{\circ \ell}(z)=0$ for all $z\in P$. 

We distinguish the one-resonant case and the $m$-resonant case with $m\ge 2$.

\sm\no {\bf 1-resonant case.} We use the same notation of \cite{BZ}.

\begin{proposition}
Let $G\in\Dif$ be one-resonant with respect to all eigenvalues $\{\l_1,\ldots,\l_n\}$, with $|\l_j|=1$, but $\lambda_j^q\ne 1$, for $j=1,\ldots, n$, and non-degenerate. Assume that $G$ is holomorphically normalizable germ, such that:
$$
\begin{cases}
\displaystyle\Re\left(\frac{a_j}{\lambda_j}\frac{1}{\Lambda(G)}\right)>0 &\hbox{for}~~j=1,\dots, s,\\
\displaystyle\Re\left(\frac{a_h}{\lambda_h}\frac{1}{\Lambda(G)}\right)<0 &\hbox{for}~~h=s+1,\dots, n,
\end{cases}
$$
for some $1\le s<n$, and let $\alpha\in\N^n$ be the generator of resonance of $\{\l_1,\dots,\l_n\}$. Then:
\begin{enumerate}
\item if $\alpha\not\in\N^{s}\times\{0\}^{n-s}$, then the unique point in a neighbourhood of the origin with orbit under $G$ converging to $O$ is the origin itself;
\sm\item if $\alpha\in\N^{s}\times\{0\}^{n-s}$, then there exists a parabolic manifold of dimension $s$ for $G$ at $O$. 
\end{enumerate}
\end{proposition}

\begin{proof}
We may assume without loss of generality that $G$ is in Poincar\'e-Dulac normal form.
Let $w$ be a point in a neighborhood of the origin with orbit under $G$ converging to the origin. Then  $\lim_{l\to \infty}\pi\circ G^{\circ l}(w) =0$, and since $\pi\circ G^{\circ l} = \Phi^{\circ l}\circ \pi$, we have $\lim_{l\to \infty}\Phi^{\circ l}(\pi(w)) =0$. Therefore two cases can occur: either $\pi(w)=0$, or $\pi(w)\ne 0$. If $\pi(w) = 0$, then $w=O$. In fact, if $\pi(w) =0$, then for every $\ell\ge 1$, we have $\|G^{\circ\ell}(w)\| = \|(\lambda_1^\ell w_1, \dots, \lambda_n^\ell w_n)\|=\|w\|$ and hence it can converge to the origin if and only if $w=O$.

Assume that $\alpha\not\in\N^{s}\times\{0\}^{n-s}$, and there exists a point $w\ne O$ in a neighborhood of the origin, whose orbit under $G$ converges to $O$. Then $\pi(w)$ lies in an attracting petal $P^+_j$ for one $j$, and so $w\in U^+_j:=\pi^{-1}(P^+_j)$. Arguing as in the proof of \cite[Proposition 4.2]{BZ}, we thus have that $w_{s+1}= \cdots= w_n=0$, and so $\pi(w) = w^\a= 0$ contradicting our hypothesis on $w$.

Assume now that $\alpha\in\N^{s}\times\{0\}^{n-s}$. Therefore, we have $G_j(z)= G_j(z_1,\dots,z_s)$, for $j=1,\dots, s$, so $G|_{\{z_{s+1}=\dots=z_n=0\}}$ is one-resonant with generator $(\a_1,\dots, \a_s)$, non-degenerate and parabolically-attracting, and hence thanks to \cite[Theorem 1.1]{BZ} we deduce that $\pi^{-1}(P^+_j)\cap \{z_{s+1}=\dots=z_n=0\}$ contains a parabolic manifold of $G$ of dimension $s$ for any attracting petal $P_j^+$. 
\end{proof}
\sm
\no {\bf $m$-resonant case.}
In the $m$-resonant case, with $m\ge 2$, it is not possible to obtain a result completely analogous to the previous proposition, because, even if it is possible to generalize \cite[Proposition 4.2]{BZ} using attracting parabolic basins along non-degenerate directions, it is not true that for any point $w$ with orbit under $G$ converging to $O$, its projection $\pi(w)$ has to lie in a basin centered along a characteristic direction, since there exist tangent to the identity germs with orbit converging to the origin but not along a characteristic direction (see \cite{Ri}). However, we have the following corollary of \cite[Theorem 1.1]{BRZ} and our Theorem 1, where we use the conditions of {\it partial parabolic attractiveness} defined in Definition \ref{def-parab_0}.

\begin{corollary}
Let $G\in\Dif$ be $m$-resonant with respect to all eigenvalues $\{\l_1,\ldots,\l_n\}$, with $|\l_j|=1$, but $\lambda_j^q\ne 1$, for $j=1,\ldots, n$, and holomorphically normalizable. Assume that $m=2$ and $G$ is attracting-nondegenerate or irregular-nondegenerate, or $m\ge 3$ and $G$ attracting-nondegenerate.  If $G$ is partially parabolically-attracting of order $1\le s<n$, and the ordered generators over $\N$ of the resonances satisfy $P^1,\dots, P^m\in\N^{s}\times\{0\}^{n-s}$, then there exists a parabolic manifold of dimension $s$ for $G$ at $O$. 
\end{corollary}

\begin{proof}
We have $G_j(z)= G_j(z_1,\dots,z_s)$, for $j=1,\dots, s$, and so $G|_{\{z_{s+1}=\dots=z_n=0\}}$ is $m$-resonant and satisfies the hypotheses of \cite[Theorem 1.1]{BRZ} or Theorem \ref{mainthm1} (if $m=2$) and so we deduce that $\{z_{s+1}=\dots=z_n=0\}$ contains a parabolic manifold of $G$ of dimension $s$. 
\end{proof}



\end{document}